\documentclass{amsart}
\usepackage{amsfonts,amssymb,amscd,amsmath,enumerate,verbatim,calc}
\usepackage[all]{xy}

\newcommand{\bF}{\mathbb{F} }
\newcommand{\M}{\mathbf{\mathcal{M}} }
\newcommand{\N}{\mathbf{\mathcal{N}} }
\newcommand{\C}{\mathbf{\mathcal{C}} }
\newcommand{\D}{\mathbf{\mathcal{D}} }
\newcommand{\T}{\mathbf{\mathcal{T}} }

\newcommand{\n}{\mathfrak{n} }
\newcommand{\m}{\mathfrak{m} }

\newcommand{\R}{\mathcal{R} }
\newcommand{\Z}{\mathbb{Z} }

\newcommand{\rt}{\rightarrow}

\newcommand{\ov}{\overline}

\newcommand{\height}{\operatorname{height}}

\newcommand{\Ass}{\operatorname{Ass}}

\newcommand{\charp}{\operatorname{char}}

\newcommand{\coker}{\operatorname{coker}}

\newcommand{\Supp}{\operatorname{Supp}}

\newcommand{\Ext}{\operatorname{Ext}}

\theoremstyle{plain}

\newtheorem{theorem}{Theorem}[section]
\newtheorem{corollary}[theorem]{Corollary}
\newtheorem{lemma}[theorem]{Lemma}
\newtheorem{proposition}[theorem]{Proposition}

\theoremstyle{definition}
\newtheorem{definition}[theorem]{Definition}
\newtheorem{s}[theorem]{}

\theoremstyle{remark}

\begin{document}

\title[Graded components of  local cohomology modules]{Graded components of  local cohomology modules over polynomial rings}
\author{Tony~J.~Puthenpurakal}
\date{\today}
\address{Department of Mathematics, IIT Bombay, Powai, Mumbai 400 076, India}

\email{tputhen@math.iitb.ac.in}
 \begin{abstract}
Let $K$ be a field and let $R = K[X_1, \ldots, X_m]$ with $m \geq 2$. Give $R$ the standard grading. Let $I$ be a homogeneous ideal of height $g$. Assume $1 \leq g \leq m -1$.
Suppose $H^i_I(R) \neq 0$ for some $i \geq 0$. We show
\begin{enumerate}
  \item $H^i_I(R)_n \neq 0$ for all $n \leq -m$.
  \item if $\Supp H^i_I(R) \neq \{ (X_1, \ldots, X_m)\}$ then $H^i_I(R)_n \neq 0$ for all  $n \in \Z$. Furthermore if $\charp K = 0$ then $\dim_K H^i_I(R)_n$ is infinite for all $n \in \Z$.
  \item $\dim_K H^g_I(R)_n$ is infinite for all $n \in \Z$.
\end{enumerate}
In fact we prove our results for $\T(R)$ where
$\T(-)$  is a large sub class of graded Lyubeznik functors.
\end{abstract}
 \maketitle
\section{introduction}
Let $K$ be a field and let $R = K[X_1, \ldots, X_m]$ with $m \geq 2$. Give $R$ the standard grading. Let $I$ be a homogeneous ideal of height $g$. The local cohomology modules $H^i_I(R)$ are graded. In this paper we study graded components of $H^g_I(R)$. There are two special cases. The first is when $I = \m = (X_1, \ldots, X_m)$. In this case we know that $H^m_\m(R) = E(m)$ where $E$ is the $R$-module consisting of inverse polynomials in  $X_1, \ldots, X_m$. In particular we have $H^m_\m(R)_n = 0$ for $n \geq -m +1$ and $\dim_K H^m_\m(R)_n$ is finite for all $n$. The other extreme case is when $I = 0$.
Then $H^0_0(R) = R$. Trivially $H^0_0(R)_n = 0$ for $n < 0$ and $\dim_K H^0_0(R)_n$ is finite for all $n \in \Z$.
Apart from these two special cases we prove
\begin{theorem}\label{main}
Let $R = K[X_1, \ldots, X_m]$ with $m \geq 2$.  Let $I$ be a homogeneous ideal of height $g$. Assume $1 \leq g \leq m -1$.
Suppose $H^i_I(R) \neq 0$ for some $i \geq 0$. Then
\begin{enumerate}[\rm (1)]
  \item $H^i_I(R)_n \neq 0$ for all $n \leq -m$.
  \item if $\Supp H^i_I(R) \neq \{ (X_1, \ldots, X_m)\}$ then $H^i_I(R)_n \neq 0$ for all  $n \in \Z$.
  \item $\dim_K H^g_I(R)_n$ is infinite for all $n \in \Z$.
\end{enumerate}
\end{theorem}
We note that Lyubeznik proved his fundamental results on local cohomology for a considerably larger class of functors which are now known as Lyubeznik functors. For definition of graded Lyubeznik functors see \ref{lyu-g}. We show
\begin{theorem}\label{main-gen}
Let $R = K[X_1, \ldots, X_m]$ with $m \geq 2$.  Let $\T$ be a graded Lyubeznik functor. Assume there exists a non-zero graded ideal $J$ of $R$ such that $\T(R)$ is $J$-torsion.  Assume $\T(R) \neq 0$. Then
\begin{enumerate}[\rm (1)]
  \item $\T(R)_n \neq 0$ for all $n \leq -m$.
  \item if $\Supp \T(R) \neq \{ (X_1, \ldots, X_m)\}$ then $\T(R)_n \neq 0$ for all  $n \in \Z$.
\end{enumerate}
\end{theorem}
We note that the assumption that there exists a non-zero graded ideal $J$ such that $\T(R)$ is $J$-torsion is satisfied by a large class of graded  Lyubeznik functors. For instance it is satisfied by
the iterated local cohomology functors $H^{i_1}_{I_1}(\cdots (H^{i_r}_{I_r}(-) \cdots))$ where $I_j$ are non-zero graded ideals in $R$. It is also satisfied by $H^i_I(-)_z$ where $I$ is a non-zero graded ideal and $z \in R$ is homogeneous.

We note that if $\Supp \T(R) = \{ (X_1, \ldots, X_m)\}$ then $T(R) = E(m)^s$ for some $s \geq 1$. In this case we know the graded components of $\T(R)$.  A natural question is whether
$\dim_K \T(R)_n$ is infinite for all $n \in \Z$ if $\Supp \T(R) \neq \{ (X_1, \ldots, X_m)\}$. We prove this in characteristic zero. We show
\begin{theorem}\label{m2}
(with hypotheses as in \ref{main-gen}) Further assume $\charp K = 0$.
If \\ $\Supp \T(R) \neq \{ (X_1, \ldots, X_m)\}$ then $\dim_K \T(R)_n$ is infinite for all $n \in \Z$.
\end{theorem}

As an easy consequence we get
\begin{corollary}
(with hypotheses as in \ref{main-gen}) Further assume $\charp K = 0$.  The following assertions are equivalent:
\begin{enumerate}[\rm (1)]
\item
$\dim_K  \T(R)_n$ is finite for some $n$.
\item
 $\Supp \T(R) = \{ (X_1, \ldots, X_m)\}$.
\end{enumerate}
\end{corollary}

We now describe in brief the contents of this paper. In section two we discuss a few preliminary results  that we need. In section three we prove a crucial Lemma that we need. In section four we give proof of Theorem \ref{main}. In the next section we give a proof of Theorem \ref{m2}.

\section{Preliminaries}
In this section we discuss some preliminary results that we need.

\s Let $R = K[X_1, \ldots, X_m]$ be standard graded. Let $I$ be a homogeneous ideal in $R$.
Then the local cohomology modules $H^i_I(R)$ are graded $R$-modules.

\s\label{lyu-g}  \textbf{Graded Lyubeznik functors:} \\
In this subsection, we define graded Lyubeznik functors.  We say $Y$ is \textit{homogeneous }closed subset of $\text{Spec}(R)$ if $Y= V(f_1, \ldots, f_s)$, where $f_i's$ are homogeneous polynomials in $R$.

We say $Y$ is a homogeneous locally closed subset of $\text{Spec}(R)$ if $Y=Y''-Y'$, where $Y', Y''$ are homogeneous closed subsets of $\text{Spec}(R)$. We have an exact sequence of graded $A_n(K)$-modules:
\begin{equation}\label{eq1} H_{Y'}^i(R) \longrightarrow H_{Y''}^i(R) \longrightarrow H_{Y}^i(R) \longrightarrow H_{Y'}^{i+1}(R).
\end{equation}

\begin{definition}
\textit{A graded Lyubeznik functor} $\mathcal{T}$ is a composite functor of the form $\mathcal{T}= \mathcal{T}_1\circ\mathcal{T}_2 \circ \ldots\circ\mathcal{T}_k$, where each $\mathcal{T}_j$ is either $H_{Y_j}^i(-)$, where $Y_j$ is a homogeneous locally closed subset of $\text{Spec}(R)$ or the kernel of any arrow appearing in (\ref{eq1}) with $Y'=Y_j'$ and $Y''= Y_j''$, where $Y_j' \subset Y_j''$ are two homogeneous  closed subsets of $\text{Spec}(R)$.
\end{definition}

\s Assume $\charp K = 0$. Consider the Weyl algebra
$$A_m(K) = K<X_1,\ldots, X_m, \partial_1, \ldots, \partial_m>.$$
 We give the grading on $A_m(K)$ as $\deg X_i = 1$ and $\deg \partial_i = -1$.  If $\T$ is a graded Lyubeznik functor then the modules  $\T(R)$  are graded holonomic $A_m(K)$-modules, see \cite{Lyu-1}. Let $\mathcal{E} = \sum_{i = 1}^{m} X_i \partial_i$ be the Eulerian operator. We say a graded $A_m(K)$-module is  generalized Eulerian if for a homogeneous $m \in M$ there exists  integer $a$ (depending on $m$) such that $(\mathcal{E} - \deg m)^am = 0$.  The  modules $\T(R)$ are generalized Eulerian, see \cite[1.7]{PS}.

 \s Assume $\charp K = p > 0$. Then we have the Frobenius map $F \colon R \rt R$ given by
 $F(r) = r^p$. There is a notion of $F$ and $F$-finite $R$-modules, see \cite{Lyu-2}.
 The  modules $\T(R)$ are graded $F$-finite $R$-modules.

\begin{theorem}\label{tame-body}[see \cite[6.1]{P}, \cite[7.2]{P2}]
 Let $R = K[X_1,\ldots, X_m]$ be standard graded with
  $K$ an infinite field.
 Let $\M = \bigoplus_{n \in \Z} \M_n$ be a graded $R$-module. If $\charp K  = p > 0$ assume  $\M$ is a graded $F_R$-finite module. If $\charp K = 0$ assume $\M$ is a graded generalized Eulerian, holonomic $A_m(K)$-module
   Then we have
\begin{enumerate}[\rm (a)]
\item
The following assertions are equivalent:
\begin{enumerate}[\rm(i)]
\item
$\M_n \neq 0$ for infinitely many $n \ll 0$.
\item
There exists $r$  such that $\M_n \neq 0$ for all $n \leq r$.
\item
$\M_n \neq 0$ for all $n \leq -m$.
\item
$\M_n \neq 0$ for some $n \leq -m$.
\end{enumerate}
\item
The following assertions are equivalent:
\begin{enumerate}[\rm(i)]
\item
$\M_n \neq 0$ for infinitely many $n \gg 0$.
\item
There exists $r$  such that $\M_n \neq 0$ for all $n \geq r$.
\item
$\M_n \neq 0$ for all $n \geq 0$.
\item
$\M_n \neq 0$ for some $n \geq 0$.
\end{enumerate}
\end{enumerate}
\end{theorem}

We also need the following
\begin{theorem}\label{rigid-body}[see \cite[6.2]{P}, \cite[7.3]{P2}]
 (with assumptions as in \ref{tame-body} with $m \geq 2$)
The following assertions are equivalent:
\begin{enumerate}[\rm(i)]
\item
$\M_n \neq 0$ for all $n \in \Z$.
\item
$\M_n \neq 0$ for some $n $ with $-m < n < 0$.
\end{enumerate}
\end{theorem}

We  also need the following result.
 Set $\ell(-) = \dim_K(-)$.
\begin{theorem}\label{len-body}[see \cite[9.4]{P}, \cite[7.4]{P2}]
 (with assumptions as in \ref{tame-body})
The following assertions are equivalent:
\begin{enumerate}[\rm(i)]
\item
$\ell(\M_n) < \infty$ for all $n \in \Z$.
\item
$\ell(\M_j) < \infty $ for some $j \in \Z $.
\end{enumerate}
\end{theorem}

Finally we need:
 \begin{lemma}\label{sushil}
 (with assumptions as in \ref{tame-body})
 Then for $l = 0, 1$
 \begin{enumerate}[\rm (1)]
  \item
  $H_l(X_m, N)$  is generalized Eulerian $A_{m-1}(K)$-module when $\charp K = 0$ and $F_{R/(X_m)}$-finite if $\charp K = p$ ; (see \cite[5.3]{PS} when \\ $\charp K = 0$ and
  \cite[1.1]{P2} when  $\charp K = p > 0$.
  \item
  If $m = 1$ then for $l = 0, 1$,
  $H_l(X_1; M)$ is concentrated in degree $0$, i.e., $H_l(X_1; M)_j = 0$ for $j \neq 0$; see \cite[5.5]{PS}  when  $\charp K = 0$ and
  \cite[1.3]{P2} when  $\charp K = p > 0$.
  \end{enumerate}
\end{lemma}

\s Let $A$ be a Noetherian ring, $I$ an ideal in $A$ and let $M$ be an $A$-module, not necessarily finitely generated.
Set
\[
\Gamma_I(M) = \{ m \in M \mid I^sm = 0 \ \text{for some} \ s \geq 0 \}.
\]
 The following result is well-known.
\begin{lemma}\label{mod-G}[with hyotheses as above]
\[
\Ass_A \frac{M}{\Gamma_I(M)} = \{ P \in \Ass_A M \mid P \nsupseteq I \}.
\]
\end{lemma}

\section{A Lemma}
In this section we prove a Lemma which is critical for us.

\begin{lemma}\label{crucial}
 Let $R = K[X_1,\ldots, X_m]$ be standard graded with
  $K$ an infinite field.
 Let $\M = \bigoplus_{n \in \Z} \M_n$ be a non-zero graded $R$-module. If $\charp K  = p > 0$ assume  $\M$ is a graded $F_R$-finite module. If $\charp K = 0$ assume $\M$ is a graded generalized Eulerian, holonomic $A_m(K)$-module.
 If $\M_n = 0$ for $n \leq -1$ then $\M = R^s$ for some $s$.
 \end{lemma}
\begin{proof}
The proof of the result both when $\charp K = p$ and when $\charp K = 0$ are analogous. We will prove the result when $\charp K = p$.
We note that as $M_{-1} = 0$, it follows from \ref{len-body} that $\ell(M_n)$ is finite for all $n \in \Z$.
We prove the result by induction on $m$. We first consider the case when $m = 1$.
We have an exact sequence
\[
0 \rt H_1(X_1, \M) \rt \M(-1) \xrightarrow{X_1} \M  \rt H_0(X_1, \M) \rt 0.
\]
Note as $\M_{-1} = 0$ we get $H_1(X_1, \M)_0 = 0$. By \ref{sushil}(2) it follows that $H_1(X_1, \M) = 0$. So $X_1$ is $\M$-regular. Also as $H_0(X_1, \M)$ is concentrated in degree zero it follows that
$\M_0 = H_0(X_1, \M)_0 $ and we have isomorphisms $\M_{i - 1} \xrightarrow{X_1} \M_i$ for all $i \geq 1$. As $\M_0$ is a finite dimensional $K$-vector space it follows that
$\M = R^s$ where $s = \ell(M_0)$.

Now assume that $m  = r\geq 2$ and the result is proved when $m = r -1$.
Let $\m = (X_1, \ldots, X_m)$ and consider $\Gamma_\m(\M)$. We note that if $\Gamma_\m(\M) \neq 0$ then it is a finite direct sum of $E(m)$. But $\M_n = 0$ for $n < 0$. This implies $E(m)=0$ which is a contradiction. Thus $\Gamma_\m(\M) = 0$. It follows that $\m \notin \Ass_R \M$. Recall $\Ass_R \M$ is a finite set. Set $V =KX_1\oplus KX_2 \oplus \cdots \oplus KX_m$. As $K$ is infinite  there exists
$$ \xi \in  V \setminus \bigcup_{P \in \Ass_R \M }P\cap V. $$
By a homogeneous linear change of variables we may assume $\xi = X_m$. Set $\N = \M/X_m \M$ and $S = R/(X_m)$.  Then $\N$ is a $F_S$-finite module.
So we have an exact sequence
$$ 0 \rt \M(-1) \xrightarrow{X_m}\M \rt \N \rt 0.$$
As $\M_n = 0$ for $n < 0$ it follows that $\N_n =0$ for $n < 0$. By induction hypotheses we have $\N = S^s$ for some $s \geq 0$.
As $\ell(\M_n)$ is finite  for all $n \in \Z$ and $\M_n = 0$ for $n < 0$; it follows from \ref{fg} that $\M$ is finitely generated as a $\R$-module. Note by graded Nakayama Lemma $\N \neq 0$. So $s \geq 1$. It follows that $\M = R^s$.
\end{proof}

The following result is needed in the proof of Lemma \ref{crucial}.
\begin{proposition}
  \label{fg} Let $R = K[X_1, \ldots, X_m]$ and let $\M$ be a graded $R$-module. Assume $\M_n = 0$ for $n \leq c$ and $\ell(\M_n) $ is finite for all $n \in \Z$. Suppose there exists $x \in R_1$ such that $\N = \M/x \M$ is a finitely generated $S = R/(x)$-module. Then $\M$ is finitely generated as a $R$-module.
\end{proposition}
\begin{proof}
Suppose $\N = (u_1, \ldots, u_l)$ and assume $u_i$ is homogeneous with $\deg u_i \leq a$. Let $D = $ the submodule of $\M$ generated by $ \M_{c+1}, \cdots, \M_{a-1}, \M_{a}$. Then $D$ is a finitely generated $R$-submodule of $\M$. We prove by induction on $n \geq c$ that $D_n = \M_n$. By construction the result holds for $n \leq a$. Now assume $n = r > a$ and the result holds when $n < r$. Let $m \in \M_r$. Let $\ov{m}$ denote the image of $m$ in $\N$. By assumption $\ov{m} = \sum_{k =1}^{l}\ov{\alpha_i} u_i$ where $\alpha_i \in R$ is homogenous. Choose $v_i \in \M$ homogeneous with $\ov{v_i} = u_i$. Notice $v_i \in D$. Set $w = \sum_{k =1}^{l}\alpha_i v_i$. Note $w \in D$ and $\ov{w} = \ov{m}$. Thus $m = w + xd$ where $d \in \M_{r-1} = D_{r-1}$. It follows that $m \in D_r$.
\end{proof}

\section{Proof of Theorem \ref{main} and \ref{main-gen}}
In this section we first give:
\begin{proof}[Proof of Theorem \ref{main-gen}]
We may assume $K$ is infinite.
Set $\M = \T(R) \neq 0$.

(1) Suppose $\M_n = 0$ for some $n \leq -m$. Then by \ref{tame-body} we have $\M_n = 0$ for $n \leq -m$. By \ref{rigid-body} we also have $\M_n = 0$  for $-m < n < 0$. Thus $\M_n = 0$ for $n < 0$.
By Lemma \ref{crucial} we get $\M = R^s$ for some $s > 0$. But $\M$ is $J$-torsion. This is a contradiction. Thus $\M_n \neq 0$ for $n \leq -m$.

(2) Suppose if possible $\M_n = 0$ for some $n \in \Z$. By \ref{rigid-body} we have $\M_n = 0$  for $-m < n < 0$. Set $\N = \M/\Gamma_\m(M)$. Then $\N$ is $F_R$-finite if $\charp K = p > 0$ and $\M$ is generalized Eulerian and holonomic if $\charp K = 0$. Thus $\Ass_R \N$ is finite. We have by \ref{mod-G} that $\m \notin \Ass_R \N$. Set $V =KX_1\oplus KX_2 \oplus \cdots \oplus KX_m$. As $K$ is infinite  there exists
$$ \xi \in  V \setminus \bigcup_{P \in \Ass_R \N }P\cap V. $$
By a homogeneous linear change of variables we may assume $\xi = X_m$.
Note we have an inclusion
\begin{equation}
  0 \rt \N(-1) \xrightarrow{X_m} \N. \tag{*}
\end{equation}
 We note that $\N_{-m +1} = 0$ (as $\N$ is a quotient of $\M$). By (*) it follows that $\N_{-m} = 0$. Iterating $\N_j = 0$ for all $j \leq -m$.
 As $\M_n = 0$  for $-m < n < 0$ and as $\N$ is a quotient of $\M$ we get that $\N_n = 0$ for  $-m < n < 0$. Thus $\N_n = 0$ for $n < 0$. By Lemma \ref{crucial} $\N = R^s$ for some $s \geq 0$.
 It follows that the exact sequence
 \[
 0 \rt \Gamma_\m(\M) \rt \M \rt \N \rt 0,
 \]
 splits. But $\M$ is $J$-torsion. So $\N = 0$. It follows that $\M = \Gamma_\m(\M)$. So $\M$ is $\m$-torsion. It follows that $\M = E(m)^r$ for some $r \geq 0$. Thus $\Supp \M = \{ \m \}$.  This is a contradiction.
 \end{proof}
Next we give
\begin{proof}[Proof of Theorem \ref{main}]
We note that (1) and (2) follow from Theorem \ref{main-gen}.\\
 (3) Set $\M = H^g_I(R)$. Suppose if possible $\ell(\M_n)$ is finite for some $n = n_0$. Then by \ref{len-body} we get that $\ell(\M_n)$ is finite for all $n \in \Z$.
 It is well-known that $\Ass H^g_I(R)  = \{ P \mid P \supseteq I \ \text{and} \ \height P = g \}$. As $1 \leq g \leq m -1$ then as argued before we may assume there exists $x \in R_1$ such that
 $x$ is $\M$-regular. After a homogeneous linear change of variables we may assume that $x = X_m$.  So we have an exact sequence
 $0 \rt \M(-1) \xrightarrow{X_m} \M \rt \C \rt 0.$ So we have
 \[
 \ell(\M_{-1}) \geq \ell(\M_{-2}) \geq \cdots \geq \ell(\M_{-n}) \geq \ell(\M_{-n-1}) \geq \cdots
 \]
 It follows that $\ell(\M_n)$ is constant for $n \ll 0$. Therefore $\C_n = 0$ for $n \ll 0$. Set $S = R/(X_m)$. Then in $\charp K = p$ we have $\C$ is $F_S$-finite. When $\charp K = 0$ we have that $\C$ is generalized Eulerian and holonomic. By \ref{tame-body} and \ref{rigid-body} we get that $\C_n = 0$ for $n < 0$.

 We claim $\C = 0$. If not then by Lemma \ref{crucial} we get $\C = S^r$ for some $r \geq 1$. But $\C$ is a submodule of $H^g_{IS}(S)$ and so is $IS$-torsion. Note $IS \neq 0$ This is a contradiction if $r \neq 0$.  Thus   $\C = 0$.
 Thus we have a bijection $ 0 \rt \M(-1) \xrightarrow{X_m} \M \rt 0$.

 We note that by \cite[1.9]{S},  $$\Ass \Ext^g_R(R/I^n, R) = \Ass \Ext^g_R(I^n/I^{n+1}, R) = \Ass \M \ \quad \text{ for all $n \geq 1$.}$$ In particular  $X_m$ is $ \Ext^g_R(R/I^n, R)$ and  $ \Ext^g_R(I^n/I^{n+1}, R)$
 regular.
 We note that we have an exact sequence $0 \rt I^{n}/I^{n+1} \rt R/I^{n+1} \rt R/I^n \rt 0$ for all $n \geq 0$. So we have an exact sequence
 $$ 0 \rt \Ext^g_R(R/I^n, R) \xrightarrow{j_n} \Ext^g_R(R/I^{n+1}, R) \rt \Ext^g_R(I^n/I^{n+1}, R), \quad \text{for all $n \geq 0$}. $$
 Set $V_n = \coker j_n$. Then note that $X_m$ is $V_n$-regular for all $n \geq 0$.

 For $n \geq 1$ we have a commutative diagram with exact rows
\[
  \xymatrix
{
 0
 \ar@{->}[r]
  & \Ext^g_R(R/I^n, R)(-1)
\ar@{->}[r]^{X_m}
\ar@{->}[d]^{j_n(-1)}
 & \Ext^g_R(R/I^n, R)
\ar@{->}[r]
\ar@{->}[d]^{j_n}
& D_n
\ar@{->}[r]
\ar@{->}[d]^{\phi_n}
&0
\\
 0
 \ar@{->}[r]
  & \Ext^g_R(R/I^{n+1}, R)(-1)
\ar@{->}[r]^{X_m}
 & \Ext^g_R(R/I^{n+1}, R)
\ar@{->}[r]
& D_{n+1}
    \ar@{->}[r]
    &0
\
 }
\]
As $X_m$ is $V_n = \coker j_n$-regular  and as $j_n$ is injective it follows that $\phi_n \colon D_n \rt D_{n+1}$ is injective for all $n$. By taking direct limits we get that $\lim D_n = 0$. It follows that $D_n  = 0$ for $n \geq 1$. By graded Nakayama's lemma we obtain $\Ext^g_R(R/I^{n}, R) = 0$ for all $n \geq 1$. This is a contradiction.
Thus $\ell(\M_n)$ is infinite for all $n \in \Z$.
\end{proof}

\section{Proof of Theorem \ref{m2}}
In this section we prove Theorem \ref{m2}. We need a few preliminaries.
Throughout this section  we assume that $\charp K = 0$.

\s \label{dual} Let $\D = A_m(K) = K<X_1, \ldots, X_m, \partial_1, \ldots, \partial_m>$ be the $m^{th}$ Weyl algebra. Consider the two subrings $R = K[X_1, \ldots, X_m]$ and $S = K[\partial_1, \ldots, \partial_m]$. We note that $E_R(K)$, the injective hull of $K$ as a $R$-module, is isomorphic to $S$. In the Weyl algebra  the $X$ and $\partial$ variables are sort of dual to each other. It is easily seen that  $E_S(K)$, the injective hull of $K$ as a $S$-module, is isomorphic to $R$. We also note that if $\M$ is a $\D$-module supported at $(\partial_1, \ldots, \partial_n)$ as an $S$-module then $\M $ is free as a $R$-module. Furthermore if $\M$ is a holonomic $\D$-module then $\Ass_S \M$ is a finite set.

We now give
\begin{proof}[Proof of Theorem \ref{m2}]
Set $\M = \T(R)$. By \ref{main-gen} we have $\M_n \neq 0$ for all $n \in \Z$. Suppose if possible $\ell(\M_n)$ is finite for some $n = n_0$. Then $\ell(\M_n)$ is finite for all $n \in \Z$, see \ref{len-body}.
Consider the exact sequence
\[
0 \rt \Gamma_\m(\M) \rt \M \rt \N \rt 0.
\]
By \ref{mod-G},  $\Ass_R \N = \Ass_R \M  \setminus \{ \m \}$.  As $\M$ is $J$-torsion it follows that if $P \in \Ass_R \M$ then $P \supseteq  J$. In particular $0 \notin \Ass_R \N$.
As argued before there exists $u \in R_1$ which is $N$-regular.

Set $\n = (\partial_1, \ldots, \partial_m) $ an ideal in $S$. We note that if $\Gamma_\n(\N) \neq 0$ then $\Gamma_\n(N) $ is a non-zero free $R$-module. Then $0 \in \Ass_R \N$; which is a contradiction. So $\Gamma_\n(\N) = 0$. Then as $\Ass_S \N$ is finite, as argued before we get $v \in S_{-1}$ which is $\N$-regular.

The injection $0 \rt \N(-1) \xrightarrow{u} \N$ yields $\ell(\N_{n-1}) \leq \ell(\N_n)$ for all $n \in \Z$.
The inclusion $0 \rt \N(+1) \xrightarrow{v} \N$ yields $\ell(\N_{n}) \leq \ell(\N_{n-1})$
for all $n \in \Z$. Thus $\ell(\N_n) = \ell(\N_{n-1})$ for all $n \in \Z$. Thus there exists $s \geq 1$ such that
$\ell(\N_n) = s$ for all $n \in \N_n$. We also have isomorphism's
 $0 \rt \N(-1) \xrightarrow{u} \N \rt 0$ and $0 \rt \N(+1) \xrightarrow{v} \N \rt 0$.

 We note that $\N$ is generated as a $\D$-module by $\N_0$. Let $\bF$ where
 $$\bF_n = \sum_{|\alpha| + |\beta| \leq n} K \underline{X}^{\alpha}\underline{\partial}^{\beta}$$
 be the Bernstein filtration of $\D$.
 Then as $\N$ is generated by $\N_0$ and as $\N_0$ is a finite dimensional $K$-vector space we get that the filtration $\Omega = \{\Omega_n \}$ where $\Omega_n = \bF_n \N_0$
 is a good filtration of $\N$.

 Notice $\Omega_n = \bF_n \N_0 = \bigoplus_{i = -n}^{n} \N_i$. So
 $\dim \Omega_n = (2n +1)s$. It follows that the Bernstein dimension of $\N$ is one.
 This is a contradiction as $\N$ being holonomic $\D = A_m(K)$-module has Bernstein dimension $m$ (which by our hypotheses is $\geq 2$).

 Thus $\ell(\M_n) $ is infinite for all $n \in \Z$.
\end{proof}

\end{document}